\documentclass[12pt, reqno]{amsart}

\usepackage[margin = 1.4in] {geometry}

\usepackage{amsmath,amssymb,amsthm,verbatim}
\usepackage{hyperref}
\usepackage[dvipsnames]{xcolor}

\newtheorem{theorem}[subsection]{Theorem}
\newtheorem{lemma}[subsection]{Lemma}

\theoremstyle{definition}

\newtheorem{remark}[subsection]{Remark}

\newcommand{\haus}{\mathcal{H}}

\newcommand{\spt}{\mathrm{spt}}

\newcommand{\reg}{\mathrm{reg}}

\newcommand{\graph}{\mathrm{graph}}

\newcommand{\eps}{\epsilon}

\newcommand{\sing}{\mathrm{sing}}

\newcommand{\R}{\mathbb{R}}

\newcommand{\del}{\partial}

\newcommand{\bbS}{\mathbb{S}}

\newcommand{\cA}{\mathcal{A}}

\newcommand{\cL}{\mathcal{L}}

\newcommand{\<}{\langle}
\renewcommand{\>}{\rangle}

\begin{document}

\title[Bernstein-type theorem for capillary graphs]{A Bernstein-type theorem for capillary graphs in a half-space}

\author{Nick Edelen}
\address{Department of Mathematics, University of Notre Dame, Notre Dame, IN 46556 USA}
\email{nedelen@nd.edu}
\author{Chao Li}
\address{Courant Institute, New York University, 251 Mercer St, New York, NY 10012, USA}
\email{chaoli@nyu.edu}
\author{Jonathan J. Zhu}
\address{Department of Mathematics, University of Washington, Seattle, WA, USA}
\email{jonozhu@uw.edu}

\begin{abstract}
We show that any entire, capillary minimal graph in a half-space must be linear in low-dimensions or, more generally, when some tangent cone at infinity does not split off a vertical line.  We also show that the regular set of any entire, capillary-minimizing hypersurface must be connected, and we discuss connections with the one-phase Bernoulli problem.
\end{abstract}

\maketitle

In this article we consider entire, graphical, capillary-minimizing hypersurfaces in a half-space $\mathbb{R}^{n+1}_+ \equiv \{ x_1 > 0 \}$. In analogy to the classical Bernstein theorem for entire graphs \cite{BoDeGi}, we show the hypersurfaces must be planar in low dimensions or under a gradient bound assumption, and in general dimensions we show the dichotomy: if the surface is not planar, then every tangent cone at infinity splits off a vertical line.  We give a geometric analytic proof, without appealing to De Giorgi-Nash-Moser theory, which may be of independent interest. Our main theorem is:

\begin{theorem}\label{thm:main1}
Take $\theta \in (0, \pi)$, and let $u \in C^2(\R^n_+) \cap C^1(\overline{\R^n_+})$ be an entire solution to the minimal surface equation with capillary boundary data:
\begin{equation}\label{eqn:main1-hyp}
\sum_{i=1}^n D_i \left( \frac{ D_i u}{\sqrt{1+ |Du|^2}} \right) = 0 \text{ on } \R^n_+, \quad \frac{-D_1 u}{\sqrt{1+ |Du|^2}} = \cos\theta \text{ on } \del \R^n_+.
\end{equation}
If either $n \leq k_*(\theta)$ or $\sup |Du| < \infty$, then $u$ is linear.
\end{theorem}

Here $k_*(\theta)$ is the smallest dimension of any singular minimizing capillary hypercone of angle $\theta$.  The regularity theory of \cite{DePhilippisMaggi} implies that a capillary-minimzing hypersurface in $\R^{n+1}$ will be regular away from some singular set of dimension $\leq n-k_*(\theta)$.  From \cite{ChEdLi, FiTsWa} (see also \cite{WangZhang}), we know that $k_* \geq 4$ for general angles, and $k_* \in \{ 5, 6, 7\}$ for angles near $0, \pi$, and $k_* = 7$ for angles near $\pi/2$.  By analogy with minimal surface theory, one may expect $k_* = 7$ for all angles, but this seems to be a difficult open question.

Geometrically, the PDE \eqref{eqn:main1-hyp} asks for the graph of $u$ to be a minimal hypersurface in $\R^{n+1}_+$ that meets the hyperplane $\del \R^{n+1}_+$ at constant angle $\theta$. Using more methods from PDE, \cite{WangWeiZhang} previously proved a version of Theorem \ref{thm:main1}, although they required both a bound on $|Du|$ and a restriction on $\theta$ in terms of $n$. 

\begin{remark}
As in the interior Bernstein problem, the critical dimension of minimal graphs is one more than the critical dimension for minimizing surfaces, due to the vertical splitting of non-trivial tangent cones at infinity (see Theorem \ref{thm:main2}).  When $\theta = \pi/2$ the dimension bound of Theorem \ref{thm:main1} is optimal, because the minimal graphs over $\R^8$ as constructed by \cite{BoDeGi} have reflection symmetry $x_1 \mapsto -x_1$. It is likely that that one may construct similar graphs for other angles.
\end{remark}

The Bernstein-type Theorem \ref{thm:main1} follows directly from a more general Theorem \ref{thm:main2} about general capillary minimizers satisfying a weak graphicality-type condition. Indeed, the above notion of graphicality over a (half-)plane \emph{orthogonal} to the barrier will imply that the graph is a capillary minimizer; this does not seem to follow naturally for graphs over other planes, such as the barrier itself. Before stating our more general theorem, we require a little bit more notation and background.

Given $\theta \in (0, \pi)$ and a set $E \subset \R^{n+1}_+$ of locally-finite perimeter, and a bounded open set $U$, we define the capillary functional
\[
\cA^\theta_U(E) = \haus^n(\del^*E \cap \R^{n+1}_+ \cap U) - \cos\theta \, \haus^n(\del^*E \cap \del \R^{n+1}_+ \cap U).
\]
We say $E$ is a minimizer for $\cA^\theta$ if for any bounded open $U$, we have $\cA^\theta_U(E) \leq \cA^\theta_U(E')$ for any $E' \subset \R^{n+1}$ of locally-finite perimeter with symmetric difference $E \Delta E' \subset \subset U$.

If $M = \del^* E \cap \R^{n+1}_+$, then by \cite{ChEdLi} $\overline{M}$ is a smooth, properly embedded submanifold-with-boundary (meeting $\del \R^{n+1}_+$ at angle $\theta$) away from some closed singular set of Hausdorff dimension $\dim\sing M\leq n-4$.  If one defines the (signed) varifold
\[
V = |\del^*E \cap \R^{n+1}_+| - \cos\theta \, |\del^* E \cap \del \R^{n+1}_+|
\]
then by the stationarity condition for $\cA^\theta$, $V$ is stationary with free-boundary in $\del\R^{n+1}_+$, and hence by the monotonicity formula
\begin{align*}
\Theta_V(x, r) &= \frac{||V||(B_r(x))}{\omega_n r^n} \\
&\equiv \frac{\haus^n(\del^* E \cap \R^{n+1}_+ \cap B_r(x)) - \cos\theta \, \haus^n(\del^* E \cap \del \R^{n+1}_+ \cap B_r(x))}{\omega_n r^n}
\end{align*}
is increasing in $r$ for every $x \in \del \R^{n+1}_+$, and strictly increasing unless $V$ is a cone centered at $x$.

From monotonicity and compactness of minimizers of $\cA^\theta$ (\cite[Theorem 2.9]{DePhilippisMaggi}), one can always blow-up or blow-down $E$ to obtain a tangent cone $E'$ (at a point or at infinity, respectively).  Specifically, given any sequence $r_i \to \infty$, after passing to a subsequence we can find a dilation-invariant minimizer $E' \subset \R^{n+1}_+$ of $\cA^\theta$ so that the rescaled minimizers $E_i := E/r_i$ limit to $E'$ in the following sense: $E_i \to E'$ in $L^1_{loc}$, $[E_i] \to [E']$ as currents, $\del [E_i] \llcorner \R^{n+1}_+ \to \del [E'] \llcorner \R^{n+1}_+$ as both varifolds and currents, $\del [E_i] \to \del [E']$ as both varifolds and currents, and $\spt (\del [E_i] \llcorner \R^{n+1}_+) \to \spt (\del [E] \llcorner \R^{n+1}_+)$ in the local Hausdorff distance.  Note that the Hausdorff convergence implies that $\del^* E' \cap \R^{n+1}_+$ will be non-empty if $\del^*E \cap\R^{n+1}_+$ is.

We are now in a position to state our main geometric theorem.
\begin{theorem}\label{thm:main2}
Take $\theta \in (0, \pi)$, and let $E \subset \R^{n+1}_+$ be a set of locally-finite perimeter which minimizes the capillary energy $\cA^\theta$.  Write $\nu_E$ for the outward unit normal of $\del^* E$.

Suppose that $\<\nu_E , e_{n+1}\> \geq 0$ at $\haus^n$-a.e. $x \in \del^* E \cap \R^{n+1}_+$, and $\del^* E \cap \R^{n+1}_+ \neq \emptyset$.  Then we have the dichotomy: either $\del^* E \cap \R^{n+1}_+$ is a capillary plane and up to rigid motion
\begin{equation}\label{eqn:main2-concl}
[E] = [\{ x_1 \cos\theta + x_{n+1} \sin \theta < 0 , x_1 > 0 \}]; 
\end{equation}
or every tangent cone $E'$ of $E$ at infinity is non-planar (i.e. does not take the form \eqref{eqn:main2-concl} up to rotation), and splits off a vertical line (i.e. takes the form $[E'] = [E'' \times \R e_{n+1}]$).
\end{theorem}

The basic idea behind Theorem \ref{thm:main2} is that for any tangent cone at infinity $E'$, the function $\<\nu_{E'}, e_{n+1}\>$ will be a non-negative $0$-homogenous Jacobi field on $M = \del^* E \cap \R^{n+1}_+$, which means that either $\<\nu_{E'} , e_{n+1}\> \equiv 0$ (in which case $E' = E'' \times \R e_{n+1}$) or the first eigenfunction of the Jacobi operator on the link $M \cap \del B_1$ is $= -(n-1)$, which can only occur if the curvature of $M$ vanishes.  In the latter case, when $E'$ is planar, monotonicity implies $E$ is planar also.

\vspace{3mm}

In Section \ref{sec:further} we make some additional remarks.  We first show how our techniques from Theorem \ref{thm:main2} can be applied to the one-phase Bernoulli problem.  We then prove the fact (indirectly related to our main theorems) that any capillary-minimizing hypersurface in $\R^{n+1}_+$ is connected, which implies (among other things) a Poincare-type inequality.  We lastly consider other graphicality-type conditions, such as when a capillary surface is graphical over the barrier plane $\del \R^{n+1}_+$, and how this condition interacts with graphicality over an orthogonal plane as in our main Theorems.

\vspace{3mm}

N.E. was supported by NSF grant DMS-2506700 and a Simons Foundation travel award.  
C.L. was partially supported by NSF grant DMS-2202343 and a Sloan Fellowship.
J.Z. was supported by NSF grant DMS-2439945. 

\section{Proofs of Main Theorems}\label{sec:proofs}

We work in $\R^{n+1}$, and write $\R^{n+1}_+ = \{ x_1 > 0 \}$, and unless otherwise stated identify $\R^n$ with the subspace $\R^{n+1} \cap \{ x_{n+1} = 0 \}$, $\R^n_+ = \R^{n+1}_+ \cap \{ x_{n+1} = 0 \}$.  We write $\bbS^n \equiv \del B_1^{n+1} \subset \R^{n+1}$ for the standard Euclidean $n$-sphere, and $\bbS^n_+ = \bbS^n \cap \R^{n+1}_+$.  If $\Sigma^k \subset \R^{n+1}$ is a submanifold with boundary, we will write $\del \Sigma$ for the submanifold boundary, and our convention will be that $\Sigma \supset \del \Sigma$ as subsets of $\R^{n+1}$.  We write $\haus^k$ for the $k$-dimensional Hausdorff measure.

We will use various notions and results from geometric measure theory.  Given $E \subset \R^{n+1}$ a set of locally-finite perimeter, we write $\del ^* E$ for the reduced boundary.  Given an oriented submanifold (possibly with boundary) $\Sigma^k$, we write $[\Sigma]$ for the induced current and $|\Sigma|$ for the induced varifold.  Given a current $T$, we write $|T|$ for the induced varifold, and $||T||$ for its mass measure.  Similarly, if $V$ is a varifold we write $||V||$ for its mass measure.  See e.g. \cite{simonGMT} for more background on geometric measure theory.

Take $E \subset \R^{n+1}_+$ minimizing $\cA^\theta$, and write $M = \del^* E \cap \R^{n+1}_+$.  We define $\reg M$ to be the set of points $x \in \overline{M}$ for which either: $x \in \R^{n+1}_+$ and $\overline{M} \cap B_r(x)$ is a smooth embedded submanifold for some $r > 0$; or $x \in \del \R^{n+1}_+$ and $\overline{M} \cap B_r(x)$ is a smooth embedded submanifold with boundary $x \in \del \reg M \subset \del \R^{n+1}_+$, whose outer conormal $\eta$ satisfies $\<\eta , -e_1\> \equiv \cos\theta$. We define $\sing M = \overline{M} \setminus \reg M$.

We first record some basic facts about entire minimizers.
\begin{lemma}\label{lem:basic}
Take $\theta \in (0, \pi)$, and let $E \subset \R^{n+1}_+$ be a set of locally-finite perimeter which minimizes $\cA^\theta$.  Write $M = \del^* E \cap \R^{n+1}_+$.  Then
\begin{enumerate}
\item$\dim(\sing M) \leq n-k_*(\theta) \leq n-4$; 
\item in $\R^{n+1} \setminus \sing M$, $\reg M$ is a properly embedded, oriented, submanifold with boundary $\del \reg M \subset \del \R^{n+1}_+$;
\item any connected component $M_1$ of $\reg M$ must have non-empty manifold boundary $\del M_1\subset \del \R^{n+1}_+$:
\item we have $\haus^n(\del^* E \cap B_r(x)) \leq c(n) r^n$ for any $x \in \R^{n+1}$, $r > 0$.
\end{enumerate}
\end{lemma}

\begin{proof}
Item 1 follows from the regularity theory of \cite{DePhilippisMaggi, ChEdLi}, Item 2 follows from the definition of $\reg M$, and Item 4 follows from the comparison
\begin{align*}
\haus^n(\del^* E \cap B_r(x)) - \omega_n r^n &\leq \cA^\theta_{B_{2r}(x)}(E \setminus B_r(x)) - \cA^\theta_{B_{2r}(x)}(E) \\
&\leq (n+1)\omega_{n+1} r^n
\end{align*}
To see Item 3, suppose towards a contradiction $M_1$ has no boundary.  Then (since $\dim(\sing M) \leq n-4$) the varifold $|M_1|$ is a stationary integral varifold in $\R^{n+1}$ supported in $\R^{n+1}_+$, and hence every tangent cone of $|M_1|$ at infinity coincides with $|\del \R^{n+1}_+|$ with some multiplicity.  But now if $E/r_i \to E'$ is any tangent cone of $E$ at infinity (for $r_i \to \infty$), this violates the varifold convergence $|\del [E/r_i] \llcorner \R^{n+1}_+| \to |\del [E'] \llcorner \R^{n+1}_+|$.
\end{proof}

\begin{lemma}\label{lem:has-plane}
Take $\theta \in (0, \pi)$, let $E \subset \R^{n+1}_+$ be a conical (i.e. dilation-invariant) 
set of locally-finite perimeter which minimizes $\cA^\theta$.  Write $M = \del^* E \cap \del \R^{n+1}_+$.

Suppose $M \neq \emptyset$, and that $\overline{E}$ contains a half-plane $P$ (with $0 \in \del P \subset \del \R^{n+1}_+$).  Then $M$ is planar, and up to rotation
\begin{equation}\label{eqn:has-plane-concl}
[E] = [\{ x_1 \cos\theta + x_{n+1}\sin\theta < 0, x_1 > 0 \}].
\end{equation}
In particular, if some connected component of $\reg M$ is planar, then \eqref{eqn:has-plane-concl} holds up to rotation.
\end{lemma}

\begin{proof}
The Lemma trivially holds for $n = 1$, as up to reflection $x_2 \mapsto -x_2$ the only minimizer is \eqref{eqn:has-plane-concl}. Proceeding by induction, we now suppose that the lemma holds with $n-1$ in place of $n$.  After rotation we can assume
\[
P = P_\phi := \left\{  \begin{array}{l l}  \{ x_{n+1} \leq 0, x_1 = 0 \} & \text{if } \phi = 0 \\ \{ x_1 \cos\phi + x_{n+1} \sin \phi = 0, x_1 \geq 0 \} & \text{if }\phi \in (0, \pi) \\ \{ x_{n+1} \geq 0, x_1 = 0 \} & \text{if } \phi = \pi \end{array} \right.
\]
for some $\phi \in [0, \pi]$.

Suppose $x \in \overline{M} \cap \del P \setminus \{0\}$.  If $E'$ is any tangent cone of $E$ at $x$, then we can write $E' = [E'' \times \R x]$ and $\overline{E''}$ contains the half-plane $P \cap (\R x)^\perp$.  Therefore by our inductive hypothesis $E'$ takes the form \eqref{eqn:has-plane-concl} up to rotation, and hence $x \in \reg M$.  We deduce that $\overline{M} = \reg M$ is regular in some neighborhood of $\del P \cap \bbS^n$.

We claim that (up to reflection $x_{n+1} \mapsto -x_{n+1}$) $P_\theta$ coincides with some connected component $M_1$ of $\reg M$ away from $\sing M$.  Let $\phi'$ be the largest angle $\in [\phi, \pi]$ for which $P_{\phi'} \subset \overline{E}$.  Since $M$ is non-empty, after reflection in $x_{n+1}$ we can assume that $\phi' < \pi$.  Now since $M$ is conical we must have $P_{\phi'} \cap \overline{M} \setminus \{0\} \neq \emptyset$.  If for some $x \in P_{\phi'} \cap \overline{M} \setminus \{0\}$ we have $x_1 > 0$, then by the interior maximum principle\footnote{Since $\del [E]$ is mass-minimizing in $\R^{n+1}_+$, one could alternately use Allard regularity to say $\del [E]$ must be regular at the touching point, and then apply the standard maximum principle for second order elliptic PDE.} \cite{SoWh} we have $M_1 \subset P_{\phi'}$ for some connected component $M_1$ of $\reg M$.  In fact by Allard regularity and the previous paragraph we must have $P_{\phi'} \setminus \{0 \} = \overline{M} \cap P_{\phi'} \setminus \{0\} = M_1\setminus \{0\}$, so by maximality of $\phi'$ the set $\overline{E}$ must meet $P_{\phi'}$ from the negative $x_{n+1}$-direction, and we get $\phi' = \theta$. 
On the other hand, if every $x \in P_{\phi'} \cap \overline{M} \setminus \{0 \}$ satisfies $x_1 = 0$, then by the previous paragraph every such $x \in \reg M$, so $\phi' = \theta$ by maximality of $\phi'$, and we can apply the Hopf lemma to get $M_1 \subset P_\theta$.

The lemma will follow if we can show that $\reg M = M_1$.  Suppose, towards a contradiction, there is some connected component $M_2$ of $\reg M \setminus M_1$.  By repeating the same argument as in our second paragraph, we have that $\reg M = M_1 = P$ in a neighborhood of $\del P \cap \bbS^n$, so $M_2$ is disjoint from $P$ near $\del P \cap \bbS^n$.

$M_2$ lives in one of the components of $\overline{\R^{n+1}_+} \setminus P_\theta$.  Let $\theta'$ be the largest angle $\in [\theta, \pi]$ for which 
\[
M_2 \subset \{ x_1 \cos\theta' + x_{n+1} \sin\theta' > 0, x_1 \geq 0 \}.
\]
If $\theta' < \pi$, then $\overline{M_2}$ touches $P_{\theta'}$, lies to one side of $P_{\theta'}$, and since $M_2$ is disjoint from $P$ near $\del P \cap \del B_1$ we have $\overline{M_2} \cap P_{\theta'} \setminus \{0\} \subset \R^{n+1}_+$.  Therefore by \cite{SoWh} (or our footnote argument) we deduce $M_2 \subset P_{\theta'}$.  However this implies $\del M_1 \cap \del M_2 \neq \emptyset$, which violates the fact that $M_1, M_2$ are different connected components.  So we must have $\theta' = \pi$.

By a similar argument, if $\theta''$ is the smallest angle in $[0, \theta]$ for which
\[
M_2 \subset \{ x_1 \cos\theta'' + x_{n+1} \sin\theta'' < 0, x_1 \geq 0 \}.
\]
then we must have $\theta'' = 0$.  We deduce that $M_2 = \emptyset$.
\end{proof}

\begin{lemma}\label{lem:cones}
Take $\theta \in (0, \pi)$, let $E \subset \R^{n+1}_+$ be a conical set of locally-finite perimeter which minimizes $\cA^\theta$, and write $\nu_E$ for the outward unit normal of $\del^* E$.  Write $M = \del^* E \cap \R^{n+1}_+$ and suppose $M \neq \emptyset$. 

Suppose that $\<\nu_E , e_{n+1}\> \geq 0$ at $\haus^n$-a.e. $x \in M$.  Then either $[E] = [E' \times \R e_{n+1}]$, or $M$ is planar and up to a rigid motion we can write
\begin{equation}\label{eqn:cones-concl}
[E] = [\{ x_1 \cos\theta + x_{n+1} \sin \theta < 0 , x_1 > 0 \}].
\end{equation}
\end{lemma}

\begin{proof}
By replacing $E$ with $\R^{n+1}_+ \setminus \overline{E}$ there is no loss in assuming $\theta \in (0, \pi/2]$.  Let $\Sigma$ be any connected component of $\reg M \cap \bbS^n$, so that $\Sigma$ is a smooth, connected, oriented, $(n-1)$-manifold in $\bbS^n$ with (non-empty) manifold boundary $\del \Sigma \subset \del \bbS^n_+$, and satisfies
\begin{equation}\label{eqn:cones-1}
\dim(\overline{\Sigma} \setminus \Sigma) \leq n-5, \quad \haus^{n-1}(\Sigma \cap B_r(x)) \leq c(n) r^{n-1} \quad \forall x, r < 1.
\end{equation}

As translation in the $e_{n+1}$ direction is a symmetry of the capillary problem, the function $f = \<\nu_E , e_{n+1}\>$ is a $0$-homogenous Jacobi field on $\reg M$, and hence solves
\[
\Delta f + |A|^2 f = 0 \text{ on } \Sigma, \quad D_\eta f = q f \text{ on } \del\Sigma
\]
where we write $q = \cot\theta A(\eta, \eta)$, and $A(v, w) \equiv A_\Sigma(v, w) = -\< \nu, D_v w \>$ for the scalar-valued second fundamental form of $\Sigma$.

Since $f \geq 0$ by assumption and $\Sigma$ is connected, the interior maximum principle and Hopf lemma imply either $f > 0$ on $\Sigma$ or $f \equiv 0$ on $\Sigma$.  Suppose $f > 0$.  By direct computation (see e.g. \cite[Theorem 1.10]{LiZhouZhu}) the function $\phi = 1 + \cos\theta\, \<\nu , e_1\>$ solves the Jacobi-like identity:
\[
\Delta \phi + |A|^2 \phi = |A|^2 \text{ on } \Sigma, \quad D_\eta \phi = q \phi \text{ on } \del\Sigma.
\]
Take $\psi$ any smooth, compactly supported function on $\Sigma$.  By integration by parts we have the identities
\[
\int_{\Sigma} \psi^2 \<\nabla \phi^2 , \nabla \log f\> = \int_\Sigma \psi^2 (|A|^2 + |\nabla \log f|^2) \phi^2 + \int_{\del\Sigma} q \psi^2 \phi^2  - \int_\Sigma \phi^2 \<\nabla \log f , \nabla \psi^2\> ,
\]
and
\[
\int_\Sigma \psi^2 |\nabla \phi|^2 = \int_\Sigma \psi^2(|A|^2 \phi^2 - |A|^2 \phi) + \int_{\del\Sigma} q \psi^2 \phi^2 - \int_\Sigma \phi \<\nabla \phi , \nabla \psi^2\>.
\]
Subtracting the two we obtain
\[
\int_\Sigma \psi^2 |\nabla \phi - \phi \nabla \log f|^2 = -\int_\Sigma \psi^2 |A|^2 \phi - \phi \<\nabla \phi - \phi \nabla \log f,  \nabla \psi^2\> ,
\]
and hence by Cauchy-Schwarz, and using the fact that $1-|\cos\theta| \leq \phi \leq 2$, we deduce
\[
\int_\Sigma |A|^2 \psi^2 \leq c(\theta)\int_\Sigma |\nabla \psi|^2,
\]
where, for instance, we may take $c(\theta)=\frac{16}{1-|\cos\theta|}$.

From the dimension and mass bounds \eqref{eqn:cones-1}, we can use a basic cutoff argument (see e.g. \cite{ss}) to find $\psi_i \in C^1_c(\Sigma)$ so that $\psi_i \to 1$ uniformly on compact subsets of $\Sigma$, and $\int_\Sigma |\nabla \psi_i|^2 \to 0$, and thereby deduce that $|A| \equiv 0$ on $\Sigma$.  So the connected component of $\reg M$ containing $\Sigma$ is planar in $\R^{n+1}_+$, and from Lemma \ref{lem:has-plane} we get that, up to rotation, $E$ must take the form \eqref{eqn:cones-concl}.

If, on the other hand, $f \equiv 0$ on every connected component $\Sigma$ of $\reg M \cap \bbS^n$, then we have $\<\nu_E , e_{n+1}\> = 0$ at $\haus^n$-a.e. $x \in \del^*E$, which implies we can decompose $[E] = [E' \times \R e_{n+1}]$.
\end{proof}

\begin{proof}[Proof of Theorem \ref{thm:main2}]
By translating $E$, we may assume without loss of generality that $0 \in \reg M$ where $M = \del^* E \cap \R^{n+1}_+$ (from Lemma \ref{lem:basic}, we know that $\reg M \cap \del \R^{n+1}_+ \neq \emptyset$).  Let $E'$ be any tangent cone of $E$ at infinity, obtained as the limit $E/r_i \to E'$.

Suppose $E'$ is planar, in the sense that \eqref{eqn:cones-concl} holds after a suitable rotation.  Define the (signed) varifold
\[
V = |\del^* E \cap \R^{n+1}_+| - \cos\theta |\del^* E \cap \del \R^{n+1}_+| ,
\]
then as in the introduction $V$ is a stationary free-boundary varifold in $\R^{n+1}_+$, and so the density $\Theta_V(0, r)$ is increasing for all $r > 0$.

Since $0 \in \reg M$ we have that $\lim_{r \to 0} \Theta_V(0, r) = (1-\cos\theta)/2$.  On the other hand, from the nature of convergence $E/r_i \to E$, the rescaled varifolds converge as 
\[
(1/r_i)_\sharp V \to V' = |\del^* E' \cap \R^{n+1}_+| - \cos\theta |\del^* E' \cap \del \R^{n+1}_+|.
\]
By our planarity assumption on $E'$, this implies $\lim_{r \to \infty} \Theta_V(0, r) = (1-\cos\theta)/2$ as well. By the monotonicity formula we deduce that $V$ is dilation-invariant, so $E$ is dilation-invariant, and hence $E = E'$ is planar. The proof of Theorem \ref{thm:main2} is completed by Lemma \ref{lem:cones}.
\end{proof}

\begin{proof}[Proof Theorem \ref{thm:main1}]
Let $\nu(x_1, \ldots, x_n)$ be the upwards unit normal of the graph of $u$ at $(x_1, \ldots, x_n)$.  Define the vector field $X(x_1, \ldots, x_{n+1}) = \nu(x_1, \ldots, x_n)$, and then $X$ is a $C^1$ vector field on $\R^{n+1}_+$ (extending continuously to $\del \R^{n+1}_+$) satisfying
\begin{equation}\label{eqn:main1-1}
|X| = 1, \quad \mathrm{div}(X) = 0 \text{ on } \R^{n+1}_+ , \quad \<X , e_1\> =  \cos\theta \text{ on } \del \R^{n+1}_+.
\end{equation}

Define $E = \{ (x_1, \ldots, x_n) : x_{n+1} < u(x_1, \ldots, x_n), x_1 > 0 \}$, and write $\nu_E$ for the outward unit normal of $E$.  If $E' \subset \R^{n+1}_+$ is any other set of locally-finite perimeter satsifying $E' \Delta E \subset\subset U$ for some precompact $U$, then using \eqref{eqn:main1-1}, the divergence theorem, and the fact that $\nu_{E'} = -e_1$ at $\haus^n$-a.e. $x \in \del^*E' \cap \del \R^{n+1}_+$ (and likewise for $E$), we can compute
\begin{align*}
\cA^\theta_U(E)
&= \int_{\del^* E \cap U} \<X , \nu_E d\haus^n\> \\
&= \int_{\del^* E' \cap U} \<X , \nu_{E'}\> d\haus^n + \int_{E \setminus E'} \mathrm{div}(X) d\haus^{n+1} - \int_{E' \setminus E} \mathrm{div}(X) d\haus^{n+1} \\
&= \int_{\del^* E' \cap \R^{n+1}_+ \cap U} \<X , \nu_{E'}\> d\haus^n - \cos\theta \haus^n(\del^* E' \cap \del \R^{n+1}_+ \cap U) \\
&\leq \cA^\theta_U(E').
\end{align*}
We deduce that $E$ is a local minimizer of $\cA^\theta$.

Trivially $\<\nu_E , e_{n+1}\> > 0$ on $\del E \cap \R^{n+1}_+ \equiv \del^* E \cap \R^{n+1}_+$, and so by Theorem \ref{thm:main2} we have that either $\del E \cap \R^{n+1}_+$ is planar, or any tangent cone $E'$ of $E$ at is infinity is non-planar and splits off a line as $[E'] = [E'' \times \R e_{n+1}]$.  In the former case we are done.

Suppose the latter alternative occurs.  If $n \leq k_*(\theta)$, then we note that $E''$ must be a conical minimizer of $\cA^\theta$ in $\R^n_+$, and so must be planar by our choice of $k_*(\theta)$, and hence $E'$ must also be planar, which is a contradiction.  If $\sup |Du| < \infty$, then $\<\nu_E , e_{n+1}\>\geq \eps > 0$ for some fixed $\eps > 0$, and hence $\<\nu_{E'} , e_{n+1}\> \geq \eps$ at $\haus^n$-a.e. $x \in \del^* E' \cap \R^{n+1}_+$, which means $E'$ cannot split off a vertical line, and we get another contradiction.
\end{proof}

\section{Further Remarks}\label{sec:further}

\subsection{One-phase Bernoulli}

The same techniques in our paper can apply to the one-phase Bernoulli problem, that is $u : \R^n \to \R$ non-negative solving (in a suitable sense) the problem
\[
\Delta u = 0 \text{ on } \{ u > 0 \}, \quad |D u| = 1 \text{ on } \del \{ u > 0 \},
\]
which can be thought of as the zero-angle limit of the capillary problem (see e.g. \cite{ChEdLi} and the references therein). 
Previously, \cite{EnFeYu} have proven a Bernstein-type theorem for entire viscosity solutions of one-phase Bernoulli whose free-boundary is a continuous graph, showing that in low dimensions the only such solutions will be linear.

Here we show an analog of the dichotomy of Theorem \ref{thm:main2} for minimizers of the one-phase problem in any dimension.  A function $u \in W^{1,2}_{loc}(\R^n)$ is called a minimizer of the one-phase problem if it minimizes on compact subsets the Alt-Caffarelli functional
\[
J(u) = \int_{\{ u > 0 \}} \left(|Du|^2 + 1\right) dx.
\]
Let us write $k_*(0)$ for the smallest dimension in which a non-linear $1$-homogeneous minimizer of the Alt-Caffarelli functional exists; from \cite{JeSa, DeJe} we know that $k_*(0) \in \{ 5, 6, 7 \}$.

\begin{theorem}\label{thm:bernoulli}
Let $u : \R^n \to \R$ be an entire, non-zero minimizer of the Alt-Caffarelli functional, and suppose that $D_n u \geq 0$ on $\{ u > 0 \}$.  Then either $u$ is linear, or every tangent cone at infinity is non-linear and is invariant in the $e_n$ direction.  In particular, if $n \leq k_*(0)$ then $u$ is linear.
\end{theorem}

We emphasize that the last conclusion (linearity in low dimensions) was first proven by \cite{EnFeYu}.

\begin{proof}
By Weiss monotonicity it will suffice to assume $u$ is $1$-homogenous.  Write $\Omega = \{ u > 0 \}$, $\Omega_0 = \Omega \cap \bbS^{n-1}$.  Since the function $f = D_n u$ is a $0$-homogenous Jacobi field generating translation in the $e_n$ direction, $f$ satisfies
\[
\Delta f = 0 \text{ on } \Omega_0, \quad D_\eta f + H f = 0 \text{ on } \reg \del \Omega_0.
\]
Here $H$ is the mean curvature scalar of $\del\Omega_0$ with respect to the outwards normal $\eta$, chosen so that $H \leq 0$ (which holds since $u$ is entire).  By e.g. \cite{EdSpVe}, $\Omega_0$ is connected, and so since $f \geq 0$ by the maximum principle and the Hopf lemma we have either $f > 0$ or $f \equiv 0$ on $\Omega_0 \cup \reg\del\Omega_0$.

If $f \equiv 0$ then $u$ is invariant in the $e_n$ direction.  Suppose $f > 0$.  By direct computation, the function $\phi = |Du|$ satisfies\footnote{Strictly speaking, it may be more technically correct to take $\phi = \sqrt{|Du|^2 + \eps}$ for $\eps > 0$.  Then $\phi$ will be smooth and satisfy
\[
\phi \Delta \phi \geq \frac{1}{n-1} |D^2 u|^2 \text{ in } \Omega_0, \quad D_\eta \phi + (1+\eps)^{-1} H\phi = 0 \text{ on } \del\Omega_0.
\]
Using this $\phi$ instead of $|Du|$ in the computations of Lemma \ref{lem:cones}, and then taking $\eps \to 0$, will still give \eqref{eqn:bernoulli-1}.}
\[
\phi \Delta \phi \geq \frac{1}{n-1} |D^2 u|^2 \text{ on } \Omega_0, \quad D_\eta \phi + H \phi = 0 \text{ on } \reg \del\Omega_0.
\]
Since $u$ is an entire minimizer, $\phi \equiv |Du| \leq 1$.

By an essentially verbatim computation as in Lemma \ref{lem:cones}, for any $\psi$ supported away from $\sing\del\Omega_0$, we have
\begin{equation}\label{eqn:bernoulli-1}
\frac{1}{n-1} \int_{\Omega_0} \psi^2 |D^2 u|^2 \leq \int_{\Omega_0} \psi^2 \phi \Delta \phi \leq 2 \int_{\Omega_0} \phi^2 |\nabla \psi|^2 \leq 2 \int_{\Omega_0} |\nabla \psi|^2.
\end{equation}
Now $\Omega_0$ is an $(n-1)$-dimensional domain in $\bbS^{n-1}$ with $\dim(\sing \del\Omega_0) \leq n-6$, and so by a capacity argument as before we can choose a sequence of $\psi_i \to 0$ uniformly on compact subsets of $\Omega_0$, and with $\int_{\Omega_0} |\nabla \psi_i|^2 \to 0$.  We deduce $D^2 u = 0$, so $u$ is linear, and since $u$ is minimizing we must have $u(x) = \<x,a\>_+$ for some unit vector $a$ satisfying $\<a , e_n\> \geq 0$.
\end{proof}

\subsection{Connectedness}

Though not required for our main theorems, we point out here that the regular set of any entire minimizer $E \subset \R^{n+1}_+$ of $\cA^\theta$ is connected.  Of course if $n \leq k_*(\theta) - 1$ then $E$ is simply planar.

\begin{theorem}\label{thm:connected}
Take $\theta \in (0, \pi)$, let $E \subset \R^{n+1}_+$ be a set of locally-finite perimeter which minimizes $\cA^\theta$, and write $M = \del^* E \cap \R^{n+1}_+$.  If $M \neq \emptyset$, then $\reg M$ is connected and $\del \reg M \subset \del\R^{n+1}_+$ is non-empty.
\end{theorem}

We remark that \cite{NaffZhu} proved a connectedness theorem for smooth hypersurfaces in the half-sphere $\bbS^n_+$ without any boundary conditions in $\del \bbS^n_+$, which says that any two such surfaces must either coincide or intersect in $\del \bbS^n_+$.  The proof of Theorem \ref{thm:connected} will more or less follow \cite{NaffZhu}, but with some modifications in the capillary minimizing setting to rule out different connected components of $M$ meeting at a (possibly singular) boundary point. A similar argument in the smooth setting should show the following improved version of \cite{NaffZhu} for properly embedded hypersurfaces with boundary $\Sigma_1,\Sigma_2$ in the hemisphere $\mathbb{S}^n_+$: Either $\Sigma_i$ are both totally geodesic, or they intersect in the interior of the hemisphere.

\begin{proof}
The key insight is \cite[Proposition 14]{NaffZhu} which characterizes any (possibly singular) stable minimal hypersurface $\Sigma$ in $\bbS^n_+$ (with fixed boundary in $\del \bbS^n_+$) as planar. 

First assume that $E$ is dilation-invariant.  Since the theorem is trivially true when $n \leq 3$ by monotonicity and the classification of cones \cite{ChEdLi}, let us assume $n \geq 4$.  Suppose there are two connected components $M_1, M_2 \subset \reg M \cap \bbS^n$. Recall that our convention is that each manifold boundary $\partial M_i \subset M_i$. 
Write $\Gamma = \del M_1 \subset \del \bbS^n_+$, and note by Lemma \ref{lem:basic} that $\Gamma$ is a smooth, oriented $(n-2)$-manifold satisfying
\begin{equation}\label{eqn:connect-1}
\dim(\overline{\Gamma}\setminus \Gamma) \leq n-5, \quad \haus^{n-2}(\Gamma) \leq c(n, \theta).
\end{equation}
So $\overline{\Gamma}$ divides $\del \bbS^n_+$ into two open connected components $S_\pm$.  Since both $\del M_1, \del M_2$ must be non-empty and disjoint, we can assume that $\del M_2 \subset S_+$.  Moreover, after replacing $E$ with $\R^{n+1}_+ \setminus E$ as necessary we can assume $\overline{E} \cap \del \bbS^n_+$ coincides with $\overline{S_+}$ in a neighborhood of $\Gamma$.

From \eqref{eqn:connect-1} $[\Gamma]$ defines an integral $(n-2)$-current in $\del \bbS^n_+$ with zero boundary.  Let $\Sigma$ be a choice of $(n-1)$-current which achives the infimum
\[
\inf \{ ||\Sigma||(\overline{\bbS^n_+}) : \del \Sigma = [\Gamma], \spt \Sigma \subset \overline{E} \cap \bbS^n \} 
\]
By standard compactness results for currents, some $\Sigma$ achieving the above infimum will exist.

If $\spt \Sigma \cap \bbS^n_+$ meets some connected component $M_3 \subset \reg M \cap \bbS^n_+$, then by the maximum principle \cite[Remark 2]{SoWh} we have $M_3 \subset \spt \Sigma$.  But now $\Sigma$ is stable for deformations staying inside $\overline{E}$, and hence $M_3$ is one-sided stable (and therefore stable) 
in $\bbS^n_+$ for deformations fixing $\del \bbS^n_+$.  From \cite{NaffZhu} we get that $M_3$ is planar, and so by Lemma \ref{lem:has-plane} $\reg M$ is connected, a contradiction.

Conversely, if $\spt \Sigma \cap \bbS^n_+$ is non-empty and disjoint from $\reg M \cap \bbS^n_+$, then $\spt \Sigma \cap \bbS^n_+$ is a smooth, stable minimal surface in $\bbS^n_+$ away from an $(n-8)$-dimensional singular set, and we can again apply \cite{NaffZhu} to get that $\spt \Sigma$ is planar, and hence by Lemma \ref{lem:has-plane} $\reg M$ is connected, which is again a contradiction.

So the only remaining possibility is $\spt \Sigma \cap \bbS^n_+ = \emptyset$, i.e. $\spt \Sigma \subset \del \bbS^n_+$.  Since $\del \Sigma = [\Gamma]$, the constancy theorem implies $\Sigma = n_+[S_+] + n_- [S_-]$ for $n_\pm \in \mathbb{Z}$ satisfying $|n_+ - n_-|= 1$.  By minimality of $||\Sigma||$ we deduce $\spt \Sigma = \overline{S_\pm} \subset \overline{E}$, and by our choice of $E$ we must have $\spt \Sigma = \overline{S_+} \subset \overline{E} \cap \bbS^n_+$.  But this contradicts the fact that $S_+ \supset \del M_2 \neq \emptyset$, since $S_+$ is relatively open in $\del \bbS^n_+$ while $\overline{E}$ locally lies to one side of $\del M_2$ in $\del \bbS^n_+$. 

This proves Theorem \ref{thm:connected} for conical $E$.  For general $E$, let $E/r_i \to E'$ be any tangent cone of $E$ at infinity, and write $M' = \del^* E' \llcorner \R^{n+1}_+$.  Since $\del [E/r_i] \llcorner \R^{n+1}_+ \to \del [E'] \llcorner \R^{n+1}_+$ as currents, varifolds, and in the local Hausdorff distance, we deduce that if $M_1 \subset \reg M$ satisfies $|M_1| \neq 0$, then $|M_1/r_i| \to |M_1'| \neq 0$ for some subset $M_1' \subset M'$.  Therefore, if $M_1$ is any connected component of $\reg M$, then $|M_1/r_i| \to |M_1'|$ for $|M_1'|$ being stationary in $\R^{n+1}_+$, and hence by the constancy theorem and connectedness of $\reg M'$ we get $|M_1/r_i| \to |M_1'| = |\reg M'| \equiv |M'|$.  If there were two connected components $M_1, M_2 \subset \reg M$, then we would have $|M/r_i| \geq |M_1/r_i| + |M_2/r_i| \to 2 |M'|$, which is a contradiction.
\end{proof}


\begin{remark}
Theorem \ref{thm:connected} implies that $E$ is indecomposable, in the sense that if $E = E_1 \cup E_2$ for disjoint sets $E_1, E_2$, and if $||\del [E]|| = ||\del[E_1]|| + ||\del[E_2]||$, then necessarily either $[E_1] = 0$ or $[E_2] = 0$.  This follows because if $E = E_1 \cup E_2$ as above, and if $E_1, E_2 \neq \emptyset$, then by the maximum principle the regular set of $E$ would necessarily be disconnected.

Moreover, following \cite{BoGi}, indecomposability implies the Neumann-Poincare-type inequality: Take $\theta \in (0, \pi)$, let $E \subset \R^{n+1}_+ \cap B_1$ be a set of locally-finite perimeter minimizing $\cA^\theta$ in $B_1$, and write $M = \del^* E \cap \R^{n+1}_+$.  Then there is a constant $\beta(n, \theta)$ so that for any $f \in C^1(B_1)$ we have the inequality
\[
\min_{k \in \R} \left( \int_{M \cap B_\beta} |f - k|^{\frac{n}{n-1}} d\haus^n \right)^{\frac{n-1}{n}} \leq \beta^{-1} \int_{M \cap B_1} |\nabla f| d\haus^n.
\]
\end{remark}

\subsection{Other graphicality conditions}\label{sec:other-graph}

We consider here capillary minimal hypersurfaces in a half-space which are graphical over their barrier plane, rather than a half-plane orthogonal to the barrier as in Theorem \ref{thm:main1}.  This is the direct ``non-linear'' analogue of the one-phase Bernoulli problem, and in fact for capillary minimizing hypersurfaces \cite{ChEdLi} showed the rescaled graphing function well-approximates a minimizer of Alt-Caffarelli for small angles.

Although this kind of graphicality does not imply minimizing by itself, we prove an capillary version of the main result in \cite{EnFeYu}, which says that when a capillary minimal hypersurface is already graphical over the barrier, then graphicality of the free-boundary over an orthogonal plane propagates to graphicality of the entire surface (over the same orthogonal plane).

In this section we identify $\R^n = \{ x_1 = 0\}$ and $\R^{n-1} = \{ x_1 = x_{n+1} = 0\}$, so that geometrically $\{ x_1 = 0 \}$ remains the barrier hyperplane.
\begin{theorem}\label{eqn:graph-to-graph}
Let $\theta \in (0, \pi/2)$, and let $v(x_2, \ldots, x_{n+1}) : \R^n \to \R$ be a non-negative, smooth in $\{ v > 0 \}$, $C^1$ up to the boundary $\del \{v > 0 \}$, entire solution the free-boundary capillary minimal surface equation:
\begin{equation}\label{eqn:graph-to-graph-hyp1}
\sum_{i=2}^{n+1} D_i \left( \frac{D_i v}{\sqrt{1 + |Dv|^2}} \right) = 0 \text{ on } \{ v > 0 \}, \quad \frac{1}{\sqrt{1+|Dv|^2}} = \cos\theta \text{ on } \del \{ v > 0 \}.
\end{equation}
Suppose that $\del \{ v > 0 \}$ is an entire graph of over $\R^{n-1}$, i.e. 
\begin{equation}\label{eqn:graph-to-graph-hyp2}
\{ v > 0 \} = \{ (x', x_{n+1}) :   x_{n+1}>u(x'), x' \in \R^{n-1} \} , \quad  u : \R^{n-1} \to \R.
\end{equation}

Then $\graph(v) := \{ (v(x), x) : x \in \overline{\{ v > 0 \}} \}$ can be rewritten as a smooth, entire graph over the half-plane $\{ x_1 \geq 0, x_{n+1} = 0 \}$ as in the hypotheses of Theorem \ref{thm:main1}.  In particular, if $n \leq k_*(\theta)$, then $v$ is linear.
\end{theorem}

\begin{proof}
Our proof is inspired by \cite{CoMaRi}.  Let $M = \graph(v) \subset \overline{\R^{n+1}_+}$, and let $\nu$ be the unit normal of $M$ in the postive $e_1$ direction.  We first note that since $\del \{ v > 0 \}$ is $C^1$, the function $u$ is continuous.  We second note that hypothesis \eqref{eqn:graph-to-graph-hyp2} implies $\{ v > 0 \}$ (and hence $M$) is connected, and $D_{n+1} v \geq 0$ on $\del \{ v > 0 \}$.

Define
\begin{equation*}
g = \frac{\< \nu , e_{n+1}\>}{ \< \nu, e_{1}\>} \equiv-D_{n+1} v, \quad W = \frac{1}{\<\nu , e_{1}\>} \equiv \sqrt{1 + |Dv|^2},
\end{equation*}
viewed as functions on $M$.  Define the operator $\cL_W(f) = \Delta f - 2\<\nabla f, \nabla \log W\>$, where $\Delta, \nabla$ denote the Laplacian, gradient on $M$.  Since for any fixed vector $e$,  $\<\nu , e\>$ is a Jacobi field on $M$ away from $\del M$, the functions $f, W$ solve
\begin{equation}\label{eqn:graph-to-graph-1}
\cL_W(g) = 0, \quad \cL_W(W) = |A|^2 W,
\end{equation}
where $A$ is the second fundamental form of $M$

By assumption $g \leq 0$ on $\del M$.  Suppose, towards a contradiction, $g(p) > \gamma > 0$ somewhere.  Define $K_\gamma = \{ x \in M  : g(x) > \gamma \}$, then $K_\gamma$ is relatively open in $M$ and at locally-positive distance away from $\del M$.  Moreover, since $W \geq g$ we have $W > \gamma$ on $K_\gamma$.

Fix $b, \delta > 0$ sufficiently small so that
\begin{equation}\label{eqn:graph-to-graph-2}
g(p)(e^{-b p_{1}} - \delta) > \gamma,
\end{equation}
and then fix $a > 0$ small so that $an + a^2 < b^2(1-(1+\gamma^2)^{-1})/2$.  Define the functions
\[
\eta(x) = e^{-a(\sqrt{1+|x - p|^2} - 1) - b x_{1}}, \quad G(x) = g(x) (\eta(x) - \delta).
\]

Using that
\[
|\nabla |x - p|^2|^2 = 4 |\pi_{T_xM}(x - p)|^2 \leq 4 |x - p|^2, \quad \Delta |x - p|^2 = 2n,
\]
it holds on $K_\gamma$ that
\begin{align*}
\Delta \eta 
&= \eta (-a \Delta \sqrt{1+|x - p|^2} + |\nabla (a \sqrt{1+|x - p|^2} + b x_{1})|^2) \\
&\geq (-an + |\nabla(a \sqrt{ 1+ |x - p|^2} + b x_{1})|^2) \\
&\geq \eta (-an - a^2|\nabla \sqrt{1+|x - p|^2}|^2 + \frac{b^2}{2} |\nabla x_{1}|^2 )\\
&\geq \eta (-an - a^2 + \frac{b^2}{2}(1 - \<\nu, e_{1}\>^2)) \\
&\geq \eta (-an - a^2 + \frac{b^2}{2}(1 - (1+\gamma^2)^{-1})) \\
&> 0
\end{align*}

Now since $\eta - \delta \leq 1$, and $\eta - \delta < 0$ outside some ambient ball $B_R$, we have $\gamma = g \geq G$ on $\del K_\gamma$ and $G < 0$ in $K_\gamma \setminus B_R$.  From our choice \eqref{eqn:graph-to-graph-2}, $G(p) > \gamma$, and therefore $G$ attains a positive maximum $> \gamma$ somewhere in $K_\gamma$.  On the other hand, we compute
\[
\cL_W(G) = \cL_W(g) (\eta - \delta) + g \Delta \eta > 0 \text{ in } K_\gamma,
\]
which precludes $G$ attaining any maximum in $K_\gamma$.  This is a contradiction, and so we must have $g \leq 0$ on $M$.  We remark that the same argument applied to $-g$ shows that $D_{n+1} v \leq \max_{\del \{ u > 0\}} D_{n+1} v$.

If $g(p) = 0$ at some $p \in M \setminus \del M$, then from \eqref{eqn:graph-to-graph-1} we can apply the strict maximum principle to deduce $g \equiv 0$ on $M$, and hence $D_{n+1} v \equiv 0$ and $v$ is invariant in the $e_{n+1}$ direction.  However, this contradicts our graphical hypothesis \eqref{eqn:graph-to-graph-hyp2}.

We must have $g < 0$ on $M \setminus \del M$.  The function $f = \< \nu, e_{n+1}\> \equiv g W^{-1}$ is the Jacobi field on $M$ generating translations in the $e_{n+1}$ direction, and so $f$ satisfies
\[
\Delta f + |A|^2 f = 0 \text{ on } M, \quad D_\eta f = \cos\theta A(\eta, \eta) f \text{ on } \del M.
\]
Since $f < 0$ on $M \setminus \del M$, the Hopf maximum principle implies $f < 0$ on $\del M$ also.  We deduce that $\< \nu, e_{n+1} \> < 0$ on $M$, and hence $M$ is a smooth graph in the $e_{n+1}$ direction over its projection $\pi_{n+1}(M)$, where $\pi_{n+1} : \R^{n+1} \to \{ x_{n+1} = 0 \}$ denotes the orthogonal projection map.

By the same calibration argument as in the proof of Theorem \ref{thm:main1}, $E := \{ 0 < x_1 < v(x_2, \ldots, x_{n+1}) \}$ minimizes the capillary functional $\cA^\theta$ among domains in the cylinder $\pi_{n+1}^{-1}(\pi_{n+1}(M))$, and in particular among domains below $E$:
\begin{equation}\label{eqn:gtog-1}
\cA^\theta(E) \leq \cA^\theta(E') \quad \forall E' \subset E.
\end{equation}
Given any radius $R$, then by our hypothesis \eqref{eqn:graph-to-graph-hyp1} we have $B_R(\tau e_{n+1}) \subset \{ v > 0 \}$ for $\tau \gg 0$.  From \eqref{eqn:gtog-1}, we can apply the non-degeneracy estimate of \cite[Lemma 4.7]{ChEdLi} to deduce
\[
\sup_{B_{R/2}(\tau e_{n+1})} v \geq \eps R
\]
for some $\eps(n, \theta) > 0$.  By the gradient estimate of \cite{CoMaRi} (see also Remark \ref{rem:gradient-estimate}), the minimal surface equation \eqref{eqn:graph-to-graph-hyp1} becomes uniformly elliptic, and we can apply Harnack to deduce $\inf_{B_{R/2}(\tau e_{n+1})} v \geq R/c(n, \theta)$.  Using again that $\{ v > 0 \}$ is the upper level set of $v$, we deduce
\[
\pi_{n+1}(M) \supset \{ (x_1, x') : 0 \leq x_{1} \leq R/c(n, \theta), |x'| \leq R/2  \}.
\]
Taking $R \to \infty$ we get $\pi_{n+1}(M) = \{ x_1 \geq 0, x_{n+1} = 0 \}$, implying that $M$ is an entire graph over the orthogonal half-plane $\{ x_1 \geq 0, x_{n+1} = 0 \}$.
\end{proof}

\begin{remark}\label{rem:gradient-estimate}
As originally contained in \cite{CoMaRi} in a more complicated setting, the same maximum principle argument with $W$ in place of $g$ and $1-\gamma^{-2}$ in place of $1-(1+\gamma^2)^{-1}$ proves the following sharp gradient estimate: if $v(x_2, \ldots, x_{n+1}) : \R^n \to \R$ is smooth on $\{ v > 0 \}$ and $C^1$ up to $\del \{ v > 0 \}$, and solves
\[
\sum_{i=2}^{n+1} D_i \left( \frac{D_i v}{\sqrt{1+ |Dv|^2}} \right) = 0 \text{ on } \{ v > 0 \}, 
\]
then
\[
\sup_{\{ u > 0 \}} |Dv| \leq \sup_{\del\{ v > 0 \}} |Dv|.
\]
\end{remark}

\bibliography{bib}
\bibliographystyle{amsplain}

\end{document}